\documentclass[oneside,english]{amsart}
\usepackage[T1]{fontenc}
\usepackage[latin9]{inputenc}
\usepackage{amsthm}
\usepackage{amssymb}
\usepackage[dvips]{graphicx}

\numberwithin{equation}{section} 
\numberwithin{figure}{section} 
\theoremstyle{plain}
\newtheorem{thm}{Theorem}
\newtheorem{prop}[thm]{Proposition}
\newtheorem{lemma}[thm]{Lemma}

 \theoremstyle{definition}
 \newtheorem{defn}[thm]{Definition}

\newcommand{\R}{\mathbb{R}}

\newcommand{\e}{\varepsilon}
\newcommand{\por}{\text{por}}
\newcommand{\essinf}{\text{essinf}}
\newcommand{\esssup}{\text{esssup}}
\newcommand{\Q}{\mathcal{Q}}
\newcommand{\F}{\mathcal{F}}
\newcommand{\ldim}{\underline{\dim}}
\newcommand{\udim}{\overline{\dim}}
\newcommand{\supp}{\text{supp}}
\newcommand{\dist}{\text{dist}}
\newcommand{\EE}{\mathbb{E}}

\usepackage{babel}

\thanks{This work was partially supported by a Leverhulme Early Career Fellowship}


\begin{document}

\begin{abstract}
Porosity and dimension are two useful, but different, concepts that quantify the size of fractal sets and measures. An active area of research concerns understanding the relationship between these two concepts. In this article we will survey the various notions of porosity of sets and measures that have been proposed, and how they relate to dimension. Along the way, we will introduce the idea of local entropy averages, which arose in a different context, and was then applied to obtain a bound for the dimension of mean porous measures.
\end{abstract}

\title{Porosity, dimension, and local entropies: a survey}

\author{Pablo Shmerkin}

\maketitle

\section{Introduction}

A large number of concepts have been introduced over the years to
quantify the ``size'' of sets of zero Lebesgue measure, in order to
be able to distinguish among them. Among the most useful ones are
a number of ``fractal dimensions'' of sets, such as Hausdorff dimension
and box-counting (or Minkowski) dimension. Generally speaking, these
dimensions arise as critical exponents when trying to cover or pack the set
in an optimal way. There are, however, other geometric quantities
that make the notion of ``zero measure'' quantitative. Recall from
the Lebesgue density theorem that if a set $E\subset\mathbb{R}^{n}$
has positive Lebesgue measure, then it contains no ``holes'', in the
sense that for almost every $x$, if $r$ is small then one cannot find
a large part of $B(x,r)$ which is disjoint from $E$. Thus, the presence
of ``holes'' of a certain relative size at all, or many, scales, is
a quantitative notion of singularity. Making this idea precise leads
to a variety of concepts of porosity of sets. Intuitively, it seems
that both these notions are related: if a set has many large holes,
then it should be easier to cover it by small sets. Likewise, if it is
possible to cover a set with few small balls, then it must have many
holes.

In this survey we will look at the connections between porosity and
dimension, but we will concentrate on measures rather than sets. As we will see, one
can define a variety of notions of porosity and dimension of measures,
which allow to distinguish among, and quantify the degree of singularity
of, measures on Euclidean space. This is a finer study than for the
case of sets, since any given set supports a large number of measures.
Indeed, for each result we will discuss for measures, a corresponding
result for sets can be obtained as a corollary.

We adopt the point of view that fractal dimensions are more naturally
defined for measures than for sets. For example, there are dual notions
of Hausdorff and packing dimensions, but for sets the usual definitions are different in crucial ways. In particular, the definition
of packing dimension requires an extra step. For measures, the duality
of both definitions becomes transparent.

In addition to their intrinsic geometric interest, notions of porosity have found applications in a variety of areas, including the theory of quasi-conformal mappings, singular integrals, and complex dynamics. In this article, our focus will be on the key geometric problems, but examples will be briefly presented to give a flavor of the breadth of applications. We strove to keep the exposition of the main geometric concepts and methods as elementary as possible, but the examples assume familiarity with certain areas.

Another recent, but shorter, survey on porosity is \cite{Jarvenpaa10}. There the focus is on the proof of the ``large porosity'' sharp bound, while here we discuss the proof of the ``small porosity'' bound (these concepts will be explained in Section \ref{sec:connections}).

This article is organized as follows. In Section \ref{sec:dimensionandentropies}, we define the dimensions of a measure that we will need in the rest of the article, introduce the concept of local entropy averages, and state and prove the main result giving the connection between entropy averages and dimension. In Section \ref{sec:porosity} we describe porosity and mean porosity of sets and measures, and give some illustrative examples and applications. In Section \ref{sec:connections}, we explore the connection between dimension and porosity. We state the most general result on the dimension of mean porous measures, briefly discuss its history, and then give a partial proof of the ``small porosity'' bound. Finally, in Section \ref{sec:generalizations} we present some further generalizations of porosity and explain its connection to conical densities and singular integrals.

\section{Dimension of measures and local entropy averages}
\label{sec:dimensionandentropies}

\subsection{Notation}

The following table summarizes the notation to be used throughout the article.

\medskip

\begin{tabular}{|c|c|}
    \hline
  Notation & Meaning\\
  \hline
  $\mathbb{N}$ & $\{1,2,\ldots\}$ \\
  $\mathbb{N}_0$ & $\{0,1,2,\ldots\}$\\
  $[n]$ & $\{0,1,\ldots, n-1\}$ \\
  $B(x,r)$ & Open ball of center $x$ and radius $r$ \\
  $B^d$ & Unit ball $B(0,1)$ in $\R^d$ \\
  $S^{d-1}$ & Unit sphere $\{x\in\R^d:|x|=1\}$ \\
  $G(d,k)$  & Grassmanian of $k$-dimensional subspaces in $\R^d$\\
 $\mathcal{M}_d$ & Borel measures on $\R^d$\\
 $\mathcal{P}_d$ & Borel probability measures on $\R^d$\\
 $\mathcal{P}^*_d$ & Borel probability measures on the unit cube $[0,1]^d$ \\
  \hline
\end{tabular}

\medskip

If $E$ is an event which may or may not hold (or, in other words, a random variable on some measure space which takes only values $0$ and $1$), then we will denote its indicator function by $\mathbf{1}(E)$.

For notational convenience, logarithms will always be to base $2$.

\subsection{Hausdorff and packing dimensions}

Let $\mu\in\mathcal{M}_d$. If $\mu(B(x,r))\sim r^\alpha$ for some $x$ and small $r$, it is reasonable to consider $\alpha$ as a kind of dimension at the point $x$. More precisely, the {\em local
dimension of $\mu$ at $x$} is defined as\[
\dim(\mu,x)=\lim_{r\to0}\frac{\log\mu(B(x,r))}{\log r},\]
provided the limit exists. In general, one can always speak of the
{\em  upper and lower local dimensions} at a point, by taking $\limsup$ and
$\liminf$ respectively. These will be denoted $\overline{\dim}(\mu,x)$
and $\underline{\dim}(\mu,x)$.

Lebesgue measure on $\R^{d}$ has local dimension $d$ at all points. More generally, if $\mu$ is Lebesgue
measure on a $k$-dimensional immersed submanifold of $\R^{d}$, then
the local dimension of $\mu$ at all points of its support is $k$.
For fractal measures, the value (and even the existence) of the local
dimension may vary from point to point, and may take non-integer values.

The (upper and lower) local dimensions reflect the decay of mass of
small balls. However, one is often interested in some global numerical
quantity, rather than a function that depends on the point. Since
we are dealing with measures, it is convenient to ignore sets of measure
zero. Thus, in order to globalize the notion of (upper/lower) local
dimension, it is natural to take the essential supremum or essential
infimum of the local dimensions. This leads to the first of our key
definitions: Hausdorff and packing dimension of a measure.

\begin{defn}
The upper (resp. lower) \emph{packing dimension} of a measure $\mu$,
denoted $\overline{\dim}_{P}\mu$ (resp. $\underline{\dim}_{P}\mu$)
is the essential supremum (resp. infimum) of the upper local dimensions.

The upper (resp. lower) \emph{Hausdorff dimension} of a measure $\mu$,
denoted $\overline{\dim}_{H}\mu$ (resp. $\underline{\dim}_{H}\mu$),
is the essential supremum (resp. infimum) of the lower local dimensions.
\end{defn}

We note the obvious inequalities
$\underline{\dim}_{H}\mu\le\overline{\dim}_{H}\mu$
, $\underline{\dim}_{P}\mu\le\overline{\dim}_{P}\mu$ ,
$\overline{\dim}_{H}\mu\le\overline{\dim}_{P}\mu$
and $\underline{\dim}_{H}\le\underline{\dim}_{P}\mu$ . All of the
inequalities may be strict. In general there is no comparison between
$\overline{\dim}_{H}\mu$ and $\underline{\dim}_{P}\mu$ .

For completeness, we recall the relationships between the dimensions of measures we have just defined, and the Hausdorff and packing dimensions of sets. For the latter, good introductions are \cite{Falconer03} and \cite{Mattila95}. The proof of the following can be found in \cite[Propositions 10.2 and 10.3]{Falconer97}:

\begin{prop} \label{prop:dim-of-measures-in-terms-of-sets}
Let $\mu\in\mathcal{P}_d$. Then:
\begin{align*}
\udim_P(\mu) &= \inf\{ \dim_P(E): \mu(E)=1 \},\\
\ldim_P(\mu) &= \inf\{ \dim_P(E): \mu(E)>0 \},\\
\udim_H(\mu) &= \inf\{ \dim_H(E): \mu(E)=1 \},\\
\ldim_H(\mu) &= \inf\{ \dim_H(E): \mu(E)>0 \}.
\end{align*}
\end{prop}

At first sight the above result may appear rather surprising, since the Hausdorff and packing dimension of sets are defined in terms of the global structure of the set, while for measures we have followed a local approach. Nevertheless, the proof of this proposition is not very difficult. The converse is a deeper fact: for any set $E$, it is possible to find measures supported on $E$ of Hausdorff and packing dimension arbitrarily close to those of $E$:

\begin{prop} \label{prop:dim-of-sets-in-terms-of-measures}
Let $E\subset\R^d$ be a Borel set. Then:
\begin{align*}
\dim_H(E) = \sup\{ \ldim_H(\mu): \mu\in\mathcal{P}_d, \mu(E)=1\},\\
\dim_P(E) = \sup\{ \ldim_P(\mu): \mu\in\mathcal{P}_d, \mu(E)=1\}.
\end{align*}
\end{prop}

For Hausdorff dimension, the above is the classical Frostman's Lemma; see \cite[Theorem 9.8]{Mattila95}. The packing dimension version is due to Cutler \cite{Cutler95}. We underline that in this article we will not make direct use of Hausdorff of packing dimensions of sets; Proposition \ref{prop:dim-of-sets-in-terms-of-measures} can be taken as their definition. An exception is Section \ref{sec:generalizations}, where some familiarity with Hausdorff measures is helpful. See e.g. \cite[Chapter 4]{Mattila95} for its definition and main properties.

\subsection{Dyadic filtrations and measures} \label{sec:dyadic-filtration}

A very extended trick in analysis is to transfer a problem to a dyadic setting,
and then profit from the tree structure of the family of dyadic cubes. One of the aims of this article is to show that this idea can be very powerful in dimension calculations, and in
particular in bounding the dimension of porous measures. Here we start by
setting up some notation and making some basic observations.

Fix an ambient dimension $d$. We denote the half-open unit cube $[0,1)^d$ by
$Q_0$. For $n\in\mathbb{N}_0$, we let $\Q_n$ be the collection of dyadic half-open cubes of step $n$, i.e.
\[
 \Q_n=\{ I_1\times \cdots\times I_d: I_i = [j_i 2^{-n},(j_i+1)2^{-n}) \text{
with } 0\le j_i<2^n \}.
\]
Further, we let $\Q^*=\bigcup_{n=1}^\infty \Q_k$ be the collection of dyadic
cubes of all levels. Each $\Q_n$ generates a finite $\sigma$-algebra which we
denote $\F_n$. The sequence $\{\F_n\}$ is increasing, and the limit
$\sigma$-algebra (i.e. the smallest $\sigma$-algebra containing all $\F_n$) is
the Borel $\sigma$-algebra of $Q_0$. Given $Q\in\Q_n$, we let $\Q_\ell(Q)$ be the
collection of cubes in $\Q_{n+\ell}$ which are contained in $Q$. In other words,
these are the level $\ell$ dyadic subcubes of $Q$.

Let $\mu\in\mathcal{P}_d^*$. To avoid technical problems,
we always make the standing assumption that $\mu$ assigns zero mass to the
boundaries of all cubes in $\Q^*$. This is usually not a restrictive assumption,
since it is satisfied by almost every translation of $\mu$, and many of the
problems one is interested in are translation-invariant. Another way to put it
is that the dyadic frame is usually just a tool, and we can translate the frame
at our convenience.

For any $Q\in\Q^*$ with $\mu(Q)>0$ and any $\ell\ge 1$, the measure $\mu$ induces a discrete
probability measure with $2^{d\ell}$ atoms, given by the relative measures of the
$\ell$-th level subcubes of $Q$. More precisely, define
\[
 \mu_\ell^Q(R) = \frac{\mu(R)}{\mu(Q)} \quad\text{for } R\in\Q_\ell(Q).
\]

We will write $\mu^Q=\mu_1^Q$. For $\ell=1$, this process can be reversed, and provides a flexible way to construct measures
satisfying certain desired properties. Namely, suppose that for each $Q\in \Q^*$, a
probability measure $\nu^Q$ on $\Q_1(Q)$ is given. One can then construct a
measure $\mu$ by starting with a unit mass on $Q_0$, and inductively splitting
the mass of each $Q\in Q_k$ according to the distribution $\nu^Q$. For example,
suppose $n=1$ and $\nu^Q$ is the uniform measure $(1/2,1/2)$; the
resulting measure is then Lebesgue measure on $[0,1]$. More generally, if
$\nu^Q$ is independent of $Q$ (but not necessarily uniform), the resulting
measure is {\em self-similar}, as the restriction to any dyadic cube is an
affine image of the original measure. Also note that, as long as each $\nu^Q$ is
globally supported (i.e. each cube of next level inherits a positive proportion
of the mass), the resulting measure $\mu$ will also be globally supported on
$Q_0$, although it can be highly singular.

In general, if we construct a measure this way, it may happen that
$\mu^Q\neq \nu^Q$ for some $Q$. Indeed, this will happen whenever the boundary
of some cube has positive measure. Under our standing assumption, the process of
forming the conditional measures $\{\mu^Q\}$ starting with a measure $\mu$, and
of building the measure $\mu$ out of the family of conditional measures
$\{\nu^Q\}$, are inverses of each other.

The definitions of local, Hausdorff and packing dimensions are in terms of Euclidean balls. It is natural to ask what happens if one instead uses dyadic cubes. Namely, let $\Q_n(x)$ denote the (unique) cube in $\Q_n$ containing $x$. Given a measure $\mu$ supported on $Q_0$, we define the upper ``dyadic local dimension'' as
\[
\overline{\dim}_2(\mu,x) = \limsup_{n\to\infty} \frac{-\log \mu(\Q_n(x))}{n}.
\]
(Recall that $\log$ is the logarithm to base $2$.) Likewise one defines the lower dyadic local dimension $\underline{\dim}_2(\mu,x)$. The question is then: how do these quantities relate to the usual (spherical) local dimensions? The potential issue is that $x$ may be very far from the center of $\Q_n(x)$, so one cannot simply nest dyadic cubes and balls centered at $x$ to conclude that $\dim_2(\mu,x)=\dim(\mu,x)$. And indeed it can happen that both are different at some points. The good news is that this can only happen on a small set:

\begin{prop}{\cite[Theorem B.1]{KRS11}} \label{prop:dim-from-dyadic-dim}
 Let $\mu\in\mathcal{P}^*_d$. Then
\[
 \overline{\dim}_2(\mu,x) = \overline{\dim}(\mu,x) \quad\text{for }\mu\text{-a.e.} x,
\]
and likewise for the lower local dimensions. In particular,
\[
 \overline{\dim}_P(\mu) = \mu\text{-}\esssup\,\, \udim_2(\mu,x),
\]
and likewise for lower packing dimension, and upper and lower Hausdorff dimensions.
\end{prop}

\subsection{Local entropy averages and dimension}

There are a number of methods to estimate the dimension of fractal measures. In this section we describe an approach, using local entropies, that turns out to be useful to bound the dimension of porous measures, as well as in other geometric and dynamical problems.

We recall some basic definitions and properties about entropy. If $p=(p_1,\ldots,p_N)$ is a probability vector, its entropy is
\[
 H(p) = \sum_{i=1}^N -p_i\,\log(p_i).
\]
The entropy quantifies the ``amount of information'' or ``uniformness'' of the vector $p$. One has
\[
 0\le H(p) \le \log N
\]
for all probability vectors with $N$ coordinates. The minimum is attained exactly at vectors with some $p_j=1$, and the maximum is attained precisely at the uniform vector $(1/N,\ldots,1/N)$. If $\nu$ is a measure with finite support, we define the entropy $H(\nu)$ in the obvious way.

Now let $\mu\in\mathcal{P}^*_d$, and fix $\ell\in\mathbb{N}$. For a point $x\in Q_0$, we can consider the sequence of conditional measures $\mu_\ell^{\Q_n(x)}$ on dyadic cubes converging down to $x$. These are discrete measures with at most $2^{\ell d}$ atoms, so we can compute their entropy (which takes values between $0$ and $\ell d$). It is reasonable to expect that for measures with large dimension, $H(\mu^{\Q_n(x)})$ will be large often, while the opposite will happen for measures of small dimension. This turns out to be exactly the case. The following is a slight extension of \cite[Lemma 4.2]{HochmanShmerkin11}; this particular formulation is due to M. Hochman [Private Communication].
\begin{thm} \label{thm:local-entropies}
Let $\mu\in\mathcal{P}^*_d$ and $\ell\in\mathbb{N}$. For almost every $x$ we have
\[
 \lim_{n\to\infty} \frac1{n} \left(-\log\mu(\Q_n(x))-  \frac1\ell \sum_{i=1}^n H\left(\mu_\ell^{\Q_i(x)}\right) \right) = 0.
\]
In particular,
\begin{align*}
\udim(\mu,x) = \limsup_{n\to\infty} \frac1{\ell n} \sum_{i=1}^n H(\mu_\ell^{\Q_i(x)}) ,\\
\ldim(\mu,x) = \liminf_{n\to\infty} \frac1{\ell n} \sum_{i=1}^n H(\mu_\ell^{\Q_i(x)}),
\end{align*}
for $\mu$-a.e.
\end{thm}

Before giving the proof, we make some remarks on this statement. Although similar results have been obtained and applied in many contexts, Theorem \ref{thm:local-entropies} seems to be the cleanest
and most powerful manifestation of the underlying idea. At first sight, it might seem that one is replacing a simple quantity - decay of mass - with an average of a more involved expression. Nevertheless, the latter is sometimes much more convenient, for the following reasons:
\begin{enumerate}
 \item One often has some information on the geometric distribution of a measure in small balls, and would like to convert that into information on the actual measure of small balls. As we will see, this is precisely the case when one deals with porosity properties.
 \item The entropy averages in the theorem are very natural from the dynamical point of view. In fact, Theorem \ref{thm:local-entropies} can be seen as a version of the Shannon-McMillan-Breiman Theorem for general (non-invariant) measures.
\item Related to the above is the fact that averages are often convenient to deal with. For example, if one lacks any control on the measure at a small number of scales, this will have small or little effect on the average.
\item Likewise, entropy has many pleasant properties that are exploited throughout ergodic theory and can also be useful in applying Theorem \ref{thm:local-entropies}. One basic example is the concavity of the entropy function $x\log(1/x)$.
\end{enumerate}

The proof of Theorem \ref{thm:local-entropies} is a straightforward application of the Law of Large Numbers for Martingale Differences, which we now recall:

\begin{thm} \label{thm:LLNMD}
 Let $\mathcal{B}_n$ be an increasing sequence of $\sigma$-algebras on a space $X$, and let $\mathcal{B}$ be the smallest $\sigma$-algebra containing all $\mathcal{B}_n$. Let $\mu$ be a measure on $(X,\mathcal{B})$. Further, suppose that $\{f_n\}$ is a uniformly $L^2$-bounded sequence of martingale differences, i.e. $f_n$ is $\mathcal{B}_n$-measurable, $\|f\|_{L^2(\mu)}$ is uniformly bounded, and $\EE(f_{n+1}|\mathcal{B}_n)=0$ for all $n$. Then
\[
 \lim_{n\to\infty} \frac1n \sum_{i=0}^{n-1} f_i(x) = 0 \quad\text{for } \mu\text{-almost every } x.
\]
\end{thm}

For a proof, see e.g. \cite[Theorem 3 in Chapter VII.9]{Feller71}. This is indeed a generalization of the Law of Large Numbers: if $\{X_k\}$ is a sequence of $L^2$-bounded i.i.d. random variables with zero mean, then letting $\mathcal{B}_n$ be the $\sigma$-algebra generated by $X_1,\ldots, X_n$, one clearly has $\EE(X_{n+1}|\mathcal{B}_n)=\EE(X_{n+1})=0$ by independence. Generally speaking, sequences of martingale differences enjoy many of the properties of independent sequences. Note also that if $\{f_n\}$ is a sequence of martingale differences, then $S_n= f_1+\ldots+f_n$ is a martingale, and conversely if $\{S_n\}$ is a martingale, then $\{S_n-S_{n-1}\}$ is a martingale difference sequence; this explains the terminology. We observe also that if $\{ g_n\}$ is any sequence where all the $g_n$ are $\mathcal{B}_n$-measurable, then $\{ g_{n+1} - \EE(g_{n+1}|\mathcal{B}_n)\}$ is a sequence of martingale differences.


\begin{proof}[Proof of Theorem \ref{thm:local-entropies}]
 Let
\[
I_n(x) = -\log(\mu^{\Q_n(x)}(\Q_{n+\ell}(x))).
\]
Recall that $\mathcal{F}_n$ is the $\sigma$-algebra generated by $\Q_n$. We fix some $0\le j<\ell$, and consider the sequence $\{I_{n\ell+j}\}$ relative to the filtration $\{ \mathcal{F}_{n\ell+j}\}$.

We have
\begin{align*}
 \EE(I_{(n+1)\ell+j}|\mathcal{F}_{n\ell+j}) &= -\sum_{R\in\Q_\ell(\Q_{n\ell+j}(x))}  \mu^{\Q_{n\ell+j}(x)}(R)\log(\mu^{\Q_{n\ell+j}(x)}(R))\\
 &= H(\mu_\ell^{\Q_{n\ell+j}(x)}).
\end{align*}
Note that
\[
 \|I_n\|_2 \le \sup \sum_{i=1}^{2^{\ell d}} \log(p_i)^2 p_i =: C_{\ell d}< \infty,
\]
where the supremum is taken over all probability vectors $(p_1,\ldots,p_{2^{\ell d}})$. Since conditional expectations do not increase $L^2$ norms, we see that
\[
f_n(x)= I_{\ell n+j}(x) - H\left(\mu^{\Q_{\ell n+j}(x)}_\ell\right)
\]
is a sequence of uniformly $L^2$-bounded martingale differences. By Theorem \ref{thm:LLNMD},
\[
 \lim_{n\to\infty} \frac1n \sum_{i=0}^{n-1} \left(I_{i\ell+j}(x) - H(\mu^{\Q_{i\ell+j}(x)}_\ell))\right) = 0.
\]
Since
\[
\sum_{i=0}^{n-1} I_{i\ell+j}(x) = -\log\left(\prod_{i=0}^{n-1} \mu^{\Q_{i\ell+j}(x)}(\Q_{(i+1)\ell+j}(x))\right) = -\log\left(\frac{\mu(\Q_{n \ell+j}(x))}{\mu(\Q_j(x))}\right),
\]
we obtain that
\[
\lim_{n\to\infty} \frac1n \left(-\log\mu(\Q_{n\ell+j}(x)) - \sum_{i=0}^{n-1} H(\mu_\ell^{\Q_{i\ell+j}(x)})  \right) = 0.
\]
Adding over $i$ from $0$ to $\ell-1$ and rescaling the summation range concludes the proof of the first part of the theorem. The latter statements are then immediate from Proposition \ref{prop:dim-from-dyadic-dim}.
\end{proof}

\section{Porosity}
\label{sec:porosity}

\subsection{Porosity of sets}

We now turn our attention to a variety of notions of porosity, and their basic
properties. We start with the simplest concept, which is the porosity of a set.

Let $E\subset\R^d$ be any set. Given a point $x\in E$ and a radius $r>0$, we
define
\begin{equation} \label{eq:porosity-reference-ball}
\por(E,x,r) = \sup\left\{  \alpha: B(y,\alpha r) \subset B(x,r)\setminus E
\text{ for some } y \right\}.
\end{equation}
We think of $B(x,r)$ as a reference ball, and $\por(E,x,r)$ as the relative size of the
largest hole of $E$  in this reference ball. The porosity of $E$
at $x$ is then defined as
\begin{equation} \label{eq:porosity-at-point}
\por(E,x) = \liminf_{r\to 0} \por(E,x,r).
\end{equation}
Thus, if $\por(E,x)>0$, then $E$ contains holes around $x$ at {\em all} sufficiently small scales.
Finally, the porosity of $E$ is the size of the largest hole one sees at all
points:
\[
\por(E) = \inf_{x\in E} \por(E,x).
\]
If $\por(E)>0$, we will simply say that $E$ is porous.

We make some remarks on this definition.
\begin{enumerate}
 \item Clearly, $0\le \por(E)\le 1/2$. Both extreme values can be attained:
$E=\R^d$ has porosity $0$, and $E=\{0\}$ has porosity $1/2$.
 \item A hyperplane in $\R^d$ has maximal porosity $1/2$. This illustrates
two trivial but important characteristics of porosity: first, rather large sets
may still have maximal porosity. Porosity itself does not allow us to
distinguish between a hyperplane and a point. In Section \ref{subsec:k-porosity} we will outline the finer
concept of $k$-porosity, which remedies this issue. Secondly, porosity (unlike
dimension) depends on the ambient space, it is not an intrinsic property of the
metric structure of a set.
 \item One could define an alternative notion by taking the $\limsup$ in
\eqref{eq:porosity-at-point}, i.e. requiring holes only at arbitrarily small
scales. Both notions of porosity are useful and well-studied. The definition we
have followed is sometimes called {\em lower porosity}, and the one with the $\limsup$, {\em upper porosity}. There are sets of large upper porosity that nevertheless have full dimension (it is easy to construct such examples via the dyadic subdivision process). Since in this article our main focus is on nontrivial relations between the ideas of porosity and dimension, we will exclusively focus on (lower) porosity without further comment.
\end{enumerate}

What are some examples of porous sets? Heuristically, any set of zero measure which satisfies some form of self-similarity (i.e. the set is made up of smaller pieces which are similar to the whole, perhaps in a statistical sense or after some distortion) is likely to be porous. Indeed, since the set has zero measure, it has macroscopic holes around some points. The self-similar structure then propagates those holes to all points and scales.

We review some basic facts on iterated function systems (IFS's). Given strictly contractive maps $f_1,\ldots, f_m$ on $\R^d$ (an IFS), there exists a unique nonempty compact set $E$, called the {\em attractor} or {\em  invariant set}, such that
\[
 E = \bigcup_{i=1}^m f_i(E).
\]
When the maps $f_i$ are similarities, $E$ is called a {\em self-similar set}. When they are conformal, $E$ is called {\em self-conformal}. It is possible to extend the notion of attractor to IFS's with a countable family of contractions, see e.g. \cite{MauldinUrbanski96} for an introduction.

Suppose that the pieces $f_i(E)$ are mutually disjoint (this is know as the {\em strong separation condition}). If the $f_i$ are similarities, or conformal under some mild additional assumptions, it is not hard to make the above argument precise to show that $E$ is indeed porous, with a bound on the porosity that depends on the derivatives of the $f_i$ and the distances between the $f_i(E)$. We remark that the nonoverlapping of the $f_i(E)$ is crucial here; there are self-similar sets of zero measure which are {\em not} porous.

The porosity of self-similar sets was investigated in \cite{JJM02}, where it was shown that porosity allows to distinguish between self-similar sets of the same dimension. Moving away from strict self-similarity, the situation may become more complicated. In \cite{Urbanski01}, the porosity of attractors of {\em conformal infinite} iterated function systems was thoroughly studied. Several conditions for porosity, and non-porosity, were given, and applied to a range of examples, including sets defined in terms of their continued fraction expansion.

Another class of porous sets are smooth curves in $\R^2$, and more generally a smooth $k$-surface in $\R^d$ ($1\le k<d$). While this is trivial, we remark that smoothness is crucial: there exist continuous curves of positive Lebesgue measure, which are therefore not porous. We will see in Section \ref{subsec:meanporosity} that a weaker condition, mean porosity, is satisfied by the boundaries of domains satisfying certain geometric assumptions which are much weaker than smoothness.

\subsection{Porosity of measures}

As remarked earlier, our main focus will be on measures rather than sets. The concept of porosity of measures was introduced by Eckmann, J\"{a}rvenp\"{a}\"{a} and J\"{a}rvenp\"{a}\"{a} in \cite{EJJ00}.

In order to adapt the definition of porosity to measures, the key issue is how to
define an appropriate notion of ``measure-theoretic hole''. A first attempt
might be to consider balls of zero measure inside a reference ball. However, it
is quite easy to see that, if we proceeded in this way, we would end up with the
porosity of the support of the measure. In order to get a notion of porosity
that is genuinely measure-theoretical, the trick is to introduce a new
parameter $\e>0$. Roughly speaking, an ``$\e$-hole'' inside a reference ball
$B(x,r)$ will be a smaller ball $B(y,\alpha r)$ with relative mass at most $\e$.
More precisely, if $\mu$ is a measure on $\R^n$, we let
\[
\por(\mu,x,r,\e) = \sup\left\{\alpha: \exists y, \,B(y,\alpha r)\subset B(x,r)
\text{ and } \frac{\mu(B(y,\alpha r))}{\mu(B(x,r))}\le \e \right\}.
\]
As before, the porosity at a point is then obtained by taking $\liminf$.
\[
 \por(\mu,x,\e) = \liminf_{r\to 0} \por(\mu,x,r,\e).
\]
To obtain the porosity of $\mu$, we need to make two natural modifications with
respect to porosity of sets. First, we take essential infimum rather than
infimum:
\[
\por(\mu,\e) = \mu\text{-}\essinf\, \por(\mu,x,\e).
\]
Finally, we need to let $\e\to 0$:
\[
 \por(\mu) = \lim_{\e\to 0} \por(\mu,\e).
\]

Before proceeding further, we make some observations.
\begin{enumerate}
\item It is important to take limits in this order; otherwise we would end up again with the porosity of the
support of $\mu$. Also note that the limit as $\e\to 0$ does exist as $\por(\mu,\e)$ is nondecreasing in $\e$.
\item Again, we have $0\le \por(\mu)\le 1/2$. The left inequality is
trivial, but the upper bound is now slightly more involved; see \cite{EJJ00} for the proof.
\item  If $\por(\mu)>0$ then $\mu$ is singular; this is an immediate consequence of the Lebesgue density theorem.
\item One trivially has $\por(\supp(\mu))\le \por(\mu))$, since any ball
disjoint from $\supp(\mu)$ in particular is an $\e$-hole for all $\e>0$. Strict
inequality is however possible. In fact, there are porous globally supported measures. For simplicity we give the construction in $\R$; the adaptation to higher dimensions is immediate. We define the measure $\mu$ on $[0,1]$ by specifying the collection of conditional measures $\nu^Q$ as in Section \ref{sec:dyadic-filtration}. For any $Q\in\Q_n$, we let $\nu^Q$ take the value $1/n$ on the left dyadic subinterval of $Q$, and $1-1/n$ on the right-dyadic subinterval. It is clear that $\mu$ is globally supported, and it is not hard to see directly from the definition that $\por(\mu)\ge 1/4$.
\end{enumerate}

\subsection{Mean porosity}
\label{subsec:meanporosity}

The first generalization of the idea of porosity involves requiring holes not at {\em all} scales, but only at a {\em positive proportion} of scales, resulting in the notion of {\em mean porosity} . This concept was introduced by P. Koskela and S. Rohde in \cite{KoskelaRohde97}, and turns out to be a very natural one. Indeed, examples of sets which are mean porous (but not necessarily porous) abound. Before giving examples, we state the formal definition (we note that this deviates slightly from the original definition given by Koskela and Rohde; the basic principle is the same but the quantitative emphasis is different).

\begin{defn}
Let $\alpha\in [0,1/2]$ (the size of the relative holes) and let $\rho\in [0,1]$ (the proportion of scales at which holes are seen).  We say that a set $E\subset\R^n$ is {\em $(\alpha,\rho)$-mean porous} if, for any $x\in E$,
\[
\liminf_{n\to\infty} \frac1n \#\{i\in [n]: \por(E,x,r)\ge\alpha \} \ge \rho.
\]
\end{defn}

If $E$ is mean $(\alpha,\rho)$-porous for some $\alpha>0,\rho>0$, we will simply say that $E$ is mean porous. Note that a mean porous set has Lebesgue measure zero thanks, once again, to the Lebesgue density theorem.

The reason why Koskela and Rohde were interested in mean porosity was to prove sharp dimension bounds for classes of sets arising in analysis. We briefly discuss some examples.

\begin{enumerate}
 \item Given a constant $c\in (0,1)$, a {\em $c$-John domain} is a domain $\Omega\subset\R^n$ containing a distinguished point $x_0$, such that the following holds: for any point $x$ in the domain, it is possible to find a curve $\gamma:[0,l]\to\Omega$, parametrized by arc length, joining $x_0$ with $x$, and having the property that
\[
 \dist(\gamma(t),\partial \Omega) \ge \frac{t}{c}  \quad\text{for all }t\in [0,l].
\]
 Thus, a $c$-John domain is a domain where points can be connected by arcs that do not come too close to the boundary. It is easy to see (using for example a Whitney decomposition) that the boundary of a $c$-John domain is mean porous.
\item Recall that a map $f:\Omega\subset\R^d\to\R^d$ is {\em $K$-quasiconformal} if for any $x\in \Omega$,
\[
 \limsup_{r\to 0}\frac{\sup\{|f(x)-f(y):|x-y|\le r\}}{\inf \{|f(x)-f(y)|:|x-y|\ge r\}} \le K.
\]
It is proved in \cite[Corollary 3.2]{KoskelaRohde97} that if $f:B^d\to\R^d$ is quasi-conformal and H\"{o}lder-continuous (with any exponent), then $f(S^{d-1})$ is mean porous. Quasiconformal maps may be far from smooth so this is a deep generalization of the fact that smooth surfaces are porous.
\item An important open problem in complex dynamics is to characterize the rational maps of the Riemann sphere which have Julia sets of full dimension. In \cite{PrzytyckiRohde98} and \cite{PrzytyckiUrbanski01} it was shown that certain important classes of Julia sets (including most satisfying the Collett-Eckmann condition) are mean porous, thereby giving a partial solution to this problem.
\end{enumerate}

Nieminen \cite{Nieminen06} generalized the notion of mean porosity by allowing the size of the holes to depend on the scale (and to possibly go to $0$ as the scale $n$ goes to infinity). He proved a fine version of the results of Koskela and Rohde using generalized Hausdorff measures, and applied these results to obtain bounds on the size of more general John domains, and images of the sphere under a wider class of quasiconformal maps.

The extension of the definition of mean porosity to measures is due to Beliaev and Smirnov \cite{BeliaevSmirnov02}. The modifications needed are the same we used to go from porosity of sets to porosity of measures.

Fix a measure $\mu\in\mathcal{M}_d$. We then say that $\mu$ is {\em weakly mean $(\alpha,\rho,\e)$-porous} if for $\mu$-almost every $x$,
\begin{equation} \label{eq:def-mean-porosity-measures}
 \liminf_{n\to\infty} \frac1n \#\{i\in [n]:\por(\mu,x,2^{-i},\e)>\alpha\} \ge \rho.
\end{equation}
We say $\mu$ is {\em mean $(\alpha,\rho)$-porous} if it is weakly mean $(\alpha,\rho,\e)$-porous for all $\e>0$.

In other words, a measure is mean $(\alpha,\rho)$-porous if, given $\e>0$, at a typical point there is a proportion at least $\rho$ of scales at which one find $\e$-holes of relative size $\alpha$. If $\mu$ is $(\alpha,\rho)$-mean porous for some $\alpha>0$ and $\rho>0$, we will simply say that $\mu$ is mean porous. The Lebesgue density theorem still implies that mean porous measures are singular.

Examples of mean porous measures which are not porous can be easily constructed using the dyadic subdivision process. For example, divide the mass uniformly among next-level cubes at scales $2^{2j}\le n<2^{2j+1}$, and pass all the mass to one of the cubes of next level at scales $2^{2j-1}\le n<2^{2j}$.

One can wonder whether porosity of measures can be expressed in terms of the porosities of sets of positive or full measure, as is the case for dimension (recall Proposition \ref{prop:dim-of-measures-in-terms-of-sets}). This turns out to be a subtle problem. In \cite[Proposition 3.1]{EJJ00} it is shown that if $\mu\in\mathcal{P}_d$ satisfies the doubling condition (i.e. $\mu(B(x,2r))\le C\mu(B(x,r))$ for some constant $C>0$, all $r>0$ and all $\in\supp\mu$), then
\begin{equation} \label{eq:porosity-of-doubling-measures}
 \por(\mu) = \sup\{ \por(E): \mu(E)>0 \}.
\end{equation}
The doubling condition is crucial for this, as shown in \cite[Example 4]{EJJ00}. Beliaev and Smirnov in \cite[Proposition 1]{BeliaevSmirnov02} claimed that a result analogous to \eqref{eq:porosity-of-doubling-measures} is valid for mean porosity without any doubling assumption. Namely, if $\mu\in\mathcal{P}_d$ is $(\alpha,\rho)$-mean porous, then for any $\delta>0$ there is a mean $(\alpha-\delta,\rho-\delta)$-porous set $E$ with $\mu(E)>1-\delta$. However, this turned out to be wrong: Beliaev et al. constructed in \cite[Theorem 4.1]{BJJKRSS09} a porous measure which gives zero mass to {\em all} mean porous sets.

Clearly, if the support of a measure $\mu$ is $(\alpha,\rho)$-mean porous, then so is $\mu$. Thus, any measure supported on the classes of mean porous sets described earlier is automatically mean porous. Natural examples of mean porous measures whose support is not porous are harder to come by, though they are easy to construct using dyadic subdivision. A reason for this is that the definition is quite strict, requiring weak mean porosity for all $\e>0$.  With an eye on the applications, weak porosity (for some $\alpha,\rho,\e>0$)  is perhaps the more appropriate concept. However, in order for weak mean porosity to have any content, $\e$ has to be smaller than the Lebesgue measure of a ball of radius $\alpha$ (otherwise, Lebesgue measure would be weakly mean porous). We informally say that $\mu$ is {\em weakly mean porous} if it is weakly mean $(\alpha,\rho,\e)$-porous for some $\e$ which is very small compared to $\alpha^d$. For geometric simplicity we will restrict ourselves to mean porous measures in the rest of the article. Estimates on the dimension of weakly mean porous measures are given in \cite{Shmerkin11}.

\section{Connections between porosity and dimension}
\label{sec:connections}

\subsection{Results and history}

We have come to main theme of this survey: the relationship between porosity and dimension. The general question we are aiming to answer is: if a measure is (mean) porous with given parameters, how large can its dimension be? Of course, the answer will in general depend on the specific notion of porosity and dimension that is being considered. It will emerge, however, that on the dimension side upper packing dimension is the natural dimension to deal with, since one cannot obtain any better general estimates for other dimensions.

Given $\alpha \in [0,1/2]$ and $\rho\in [0,1]$, let us define
\[
 G_d(\alpha,\rho) = \sup\{\udim_P(\mu):\mu\in\mathcal{P}_d \text{ is mean } (\alpha,\rho)\text{- porous} \}.
\]
Ideally, we would like to compute $G_d(\alpha,\rho)$ explicitly, but this seems to be out of reach. Rather, research has focused on the asymptotic behavior of $G_d(\alpha,\rho)$ in the two limiting cases: as $\alpha\to 0$ (``small porosity case'') and as $\alpha\to 1/2$ (``large porosity case''). We first state the results, and then summarize the history that led to them.

\begin{thm} \label{thm:main}
There is a constant $C_d$ (depending only on the ambient dimension $d$) such that
\begin{align}
G_d(\alpha,\rho) &\le d - C_d\, \rho\,\alpha^d, \label{eq:bound-small-porosity}\\
G_d(\alpha,\rho) &\le d - \rho + \frac{C_d}{|\log(1-2\alpha)|} \label{eq:bound-large-porosity}.
\end{align}
\end{thm}

Up to the value of the constant $C_d$, the first bound is sharp as $\alpha\to 0$.  For the second bound, there are examples showing that
\[
G_d(\alpha,\rho) \ge d - \rho + \frac{C'_d\,\rho}{|\log(1-2\alpha)|}
\]
for some constant $C'_d>0$.

Historically, the analog problem for sets was studied first. Motivated by various geometrical problems, Sarvas \cite{Sarvas75}, Trocenko \cite{Trocenko81} and Martio and Vuorinen \cite{MartioVuorinen87} proved bounds for the dimension of porous sets, concluding in the bound
\[
\dim_P(E) \le d - c_d\, (\por(E))^d.
\]
See also \cite{JJKRRS10} for an extension to certain metric spaces. Mattila \cite{Mattila88} made the first contribution to the large porosity case, applying his results on conical densities to identify the limiting behavior as the porosity tends to its maximum value (See Section \ref{subsec:conical-densities}): he proved that
\[
\sup\{\dim_H(E): \por(E)\ge \alpha\} \to d-1 \quad\text{as } \alpha\to \frac12.
\]
Salli \cite{Salli91} then gave the asymptotically correct estimates for packing dimension in the large porosity case:
\[
\dim_P(E)\le d - 1 + \frac{C_d}{|\log(1-2\por(E))|}
\]
(Compare with \eqref{eq:bound-large-porosity}). Koskela and Rohde \cite{KoskelaRohde97} dealt with the upper bound on the packing dimension of mean porous sets (using a more general definition), in the small porosity case.

It turned out that the extension of these results to measures was far from a merely technical task. In the article \cite{EJJ00}, an upper bound corresponding to \eqref{eq:bound-large-porosity} was proved for porous measures satisfying the doubling condition. In \cite{JarvenpaaJarvenpaa02}, the same result is stated without the doubling assumption, but it was remarked later in \cite{JarvenpaaJarvenpaa05} that the proof in \cite{JarvenpaaJarvenpaa02} only works for Hausdorff dimension. The results in \cite{JarvenpaaJarvenpaa05}, on the other hand, imply the small porosity bound \eqref{eq:bound-small-porosity} for the Hausdorff dimension of mean porous measures.

Theorem \ref{thm:main} was stated by Beliaev and Smirnov in \cite{BeliaevSmirnov02}. Their proof relied on a reduction to the corresponding inequalities for the packing dimension of porous sets. Unfortunately, the argument that enables this reduction is incorrect, as discussed above: it is not possible to express the (mean) porosity of a general measure in terms of the (mean) porosities of sets of positive measure. However, Beliaev and Smirnov did provide the first proof for the correct upper bound for the packing dimension of mean porous sets in the large porosity case.

Finally, the proof of Theorem \ref{thm:main} was achieved in \cite{BJJKRSS09} (large porosity) and \cite{Shmerkin11} (small porosity). Independently of \cite{Shmerkin11}, general results for the dimension of porous (rather than mean porous) measures in doubling metric spaces were obtained by K\"{a}enm\"{a}ki, Rajala and Suomala \cite{KRS11} (they require a technical condition on the measure which is automatically satisfied in Euclidean spaces).

It follows from Proposition \ref{prop:dim-of-sets-in-terms-of-measures} that mean porous sets satisfy the same packing dimension bounds as mean porous measures. Thus, one can recover the earlier results of Koskela and Rohde \cite{KoskelaRohde97} and of Beliaev and Smirnov \cite{BeliaevSmirnov02} from Theorem \ref{thm:main} (although, to be precise, one cannot obtain the full strength of the estimates in \cite{KoskelaRohde97}, since they use a finer version of mean porosity).

\subsection{The small porosity case: a sketch of proof}

In this section we present the main ideas in the proof of the bound \eqref{eq:bound-small-porosity}. In fact, we are going to give a full proof of the following slightly weaker bound:
\begin{prop} \label{prop:weaker-bound}
 There is a constant $C_d>0$ such that if $\mu\in\mathcal{P}_d$ is mean $(\alpha,\rho)$-porous, then
\[
 \udim_p(\mu) \le d - \frac{C_d\,\rho\,\alpha^d}{|\log\alpha|}.
\]
\end{prop}

After the proof of this proposition, we will indicate the modifications needed to get the sharp bound \eqref{eq:bound-small-porosity}. The proof of Proposition \ref{prop:weaker-bound} consists of two parts: first, we reduce the problem to one involving a dyadic version of porosity. This is then handled using the method of locally entropy averages.

\subsubsection{Dyadic Porosity}

Fix $\mu\in\mathcal{P}_d^*$. Given $x\in\Q_0, n\in\mathbb{N}$ and $\e>0$, we let
 \[
 \por_2(\mu,x,n,\e) = \min\{\ell:\exists R\in\Q_\ell(\Q_n(x)): \mu(R) \le \e \mu(Q_n(x)) \}.
 \]
We then say that $\mu$ is {\em dyadic weakly mean-$(\ell,\rho,\e)$ porous} if for $\mu$-almost every $x$,
\begin{equation} \label{eq:def-weakly-mean-dyadic-porosity}
 \liminf_{n\to\infty} \frac1n \#\{i\in [n]: \por_2(\mu,x,n,\e)\le \ell \} \ge \rho.
\end{equation}
Finally, $\mu$ is {\em dyadic mean $(\ell,\rho)$-porous} if it is dyadic weakly mean $(\ell,\rho,\e)$-porous for all positive $\e>0$. In other words, a measure is dyadic mean $(\ell,\rho)$-porous if, for every $\e>0$, at a typical point there is a proportion $\rho$ of dyadic scales for which some dyadic subcube of relative size $2^{-\ell}$ is an $\e$-hole. A similar notion of dyadic porosity was defined by Beliaev and Smirnov in \cite[\S 6.2]{BeliaevSmirnov02}. The concept of average homogeneity, which was defined and investigated in \cite{JarvenpaaJarvenpaa05}, can be seen as a generalization of dyadic mean porosity.

The next lemma is the key in reducing spherical porosity to dyadic porosity (at the cost of losing constant factors); it is essentially \cite[Lemma 4.3]{Shmerkin11}.

\begin{lemma} \label{lem:porosity-to-dyadic-porosity}
Let $\mu\in\mathcal{P}^*_d$ be a mean $(\alpha,\rho)$-porous measure. Then for almost every $t\in [0,1/2)^d$, the measure $\frac14 \mu+t$ is dyadic mean $(\ell,2^{-d}\rho)$-porous, where
\begin{equation} \label{eq:ell-of-alpha}
\ell = \lceil|\log(\alpha/{4\sqrt{d}})|\rceil.
\end{equation}
\end{lemma}
\begin{proof}
Choose $t$ at random  uniformly in $[0,1/2)^d$, and consider the random measure $\widetilde{\mu}=\frac14\mu+t$. Let $\e>0$ and fix a point $x\in\supp\mu$ such that
\begin{equation} \label{eq:porous-point}
 \liminf_{n\to\infty} \frac1n \#\{i\in [n]: \por_2(\mu,x,n,\e)\le \ell \} \ge \rho.
\end{equation}

Write $\widetilde{x}=\frac14 x+t$. The \textit{relative position} of a point $y$ inside a dyadic cube $R$, denoted as $\text{pos}(y,R)$, is defined to be $T(y)$, where $T$ is the natural homothety mapping $R$ onto $Q_0$. Note that $\widetilde{x}\bmod 1/2$ is a random variable whose distribution is Lebesgue measure on $[0,1/2)^d$, and therefore the relative positions $\text{pos}(\widetilde{x},\Q_n(\widetilde{x}))$, $n\ge 2$, form a sequence of i.i.d. uniformly distributed random variables.

Let $U$ be the set of points in $Q_0$ at distance at least $\frac14$ from the boundary $\partial Q_0$. Since $U$ contains a cube of side $\frac12$, its Lebesgue measure is at least $2^{-d}$. Call $i$ a good scale if
\[
\por_2(\widetilde{\mu},\widetilde{x},r,\e)\ge \alpha \quad\textrm{and}\quad \text{pos}(\widetilde{x},\Q_i(\widetilde{x}))\in U.
\]
Then \eqref{eq:porous-point} and the law of large numbers imply that almost surely,
\begin{equation} \label{eq:proportion-of-porous-central-scales}
\liminf_{n\rightarrow\infty} \frac{1}{n} |\{ i\in [n] : i \text{ is a good scale }  \}| \ge 2^{-d}\rho.
\end{equation}
Note that any ball of radius $\frac14 2^{-i} \alpha$ contains a dyadic cube of side length $2^{-(i+\ell)}$, where $\ell$ is as in \eqref{eq:ell-of-alpha}. Therefore, if $i$ is a good scale then $\Q_i(\widetilde{x})$ contains a dyadic cube $R\in \Q_\ell(\Q_i(\widetilde{x}))$ satisfying
\begin{align*}
\widetilde{\mu}(R) & \le \widetilde{\mu}(B(x,\alpha \frac14 2^{-i})) \\
&\le \e\widetilde{\mu}(B(x,\frac14 2^{-i}))  \le \e\widetilde{\mu}(\Q_i(\widetilde{x})),
\end{align*}
or, in other words, $\por_2(\widetilde{\mu},\widetilde{x},n,\e)\le \ell$. In light of \eqref{eq:proportion-of-porous-central-scales}, almost surely
\[
\liminf_{n\rightarrow\infty} \frac{1}{n} \#\{i\in [n]:\por_2(\widetilde{\mu},\widetilde{x},n,\e)\le \ell\}  \ge 2^{-d}\rho.
\]
The foregoing analysis is for a fixed $x$ satisfying \eqref{eq:porous-point}. Now, since $\mu$ is mean $(\alpha,\rho)$-porous, $\mu$-a.e. point satisfies \eqref{eq:porous-point} for any $\e>0$ (taking a sequence $\e_j\to 0$), and therefore we can apply Fubini to conclude that for almost every $t$, the measure $\frac14\mu+t$ is dyadic mean $(\ell, 2^{-d}\rho)$-porous, as desired.
\end{proof}

\subsubsection{Proof of Proposition \ref{prop:weaker-bound}}

Completing the proof of Proposition \ref{prop:weaker-bound} is now only a matter of applying the local entropy averages method (Theorem \ref{thm:local-entropies}).

\begin{proof}[Proof of Proposition \ref{prop:weaker-bound}]
Since porosity is a local concept, we can assume without loss of generality that $\mu\in\mathcal{P}^*_d$. By Lemma \ref{lem:porosity-to-dyadic-porosity}, it is enough to show that if $\mu$ is dyadic mean $(\ell,\rho)$-porous, then
\begin{equation} \label{eq:dyadic-weaker-bound}
\udim_P(\mu) \le d - \frac{C'_d\, 2^{-\ell d}\,\rho}{\ell},
\end{equation}
for some constant $C'_d>0$ (note from \eqref{eq:ell-of-alpha} that $\alpha$ and $2^{-\ell}$ are comparable up to multiplicative constants which only depend on $d$).

Suppose then that $\mu$ is dyadic mean $(\ell,\rho)$-porous, and fix $\e>0$ arbitrarily small. Let $A_\e$ be the set of $x$ such that
\[
\liminf_{n\to\infty} \frac1n \#\{i\in [n]:\por_2(\widetilde{\mu},\widetilde{x},n,\e)\le \ell \} \ge \rho.
\]
Then $\mu(A_\e)=1$ by definition of mean porosity.

Let $H_\e$ the maximum possible entropy over all probability vectors $(p_1,\ldots, p_{2^{d\ell}})$, where some $p_i\le \e$. Although it is not hard to calculate $H_\e$ explicitly, we only require the easy fact that
\begin{equation} \label{eq:limit-entropy}
\lim_{\e\to 0} H_\e=H_0 = \log(2^{d\ell}-1).
\end{equation}
Now if $x\in A_\e$, then it follows from our assumption that
\[
\liminf\frac1n \sum_{i=1}^n \mathbf{1}\left(H(\mu_\ell^{\Q_i(x)}) \le H_\e\right) \ge \rho.
\]
It follows that if $x\in A_\e$ then
\[
\limsup_{n\to\infty} \frac{1}{\ell n} \sum_{i=1}^n H(\mu_\ell^{\Q_i(x)}) \le \frac{1}{\ell}(\rho H_\e +(1-\rho)d\ell).
\]
Let $A=\bigcap_{\e>0} A_\e = \bigcap_{n=1}^\infty A_{1/n}$ (since the $A_\e$ are nested). Then $\mu(A)=1$ and, by \eqref{eq:limit-entropy}, if $x\in A$ then
\begin{align*}
\limsup_{n\to\infty} \frac{1}{\ell n} \sum_{i=1}^n H(\mu_\ell^{\Q_i(x)}) &\le \frac{1}{\ell} \left(\rho \log(2^{d\ell}-1) +(1-\rho)d\ell\right) \\
&=  d + \frac{\log(1-2^{-d\ell})\rho}{\ell} \\
&\le d -  \frac{2^{-d\ell}\rho}{\ell},
\end{align*}
using that $|\log(1-x)|\ge x$ for $0< x<1$. Since this holds for $\mu$-almost every $x$, we conclude from Theorem \ref{thm:local-entropies} that \eqref{eq:dyadic-weaker-bound} holds (with $C'_d=1$), completing the proof.
\end{proof}

\subsubsection{The sharp upper bound}

In order to get rid of the logarithmic factor in Proposition \ref{prop:weaker-bound}, one needs to use a more general version of local entropy averages (Theorem \ref{thm:local-entropies}), in which cubes are subdivided into cubes of many different sizes, in a way that is allowed to depend on the cube:
\begin{itemize}
\item If a cube $Q$ contains an $\e$-hole $R\in\Q_\ell(Q)$, then we split $Q$ into a family of cubes consisting of $R$ and , for each $1\le j\le \ell$, all dyadic subcubes of $Q$ of level $j$ which do not contain $R$. See Figure \ref{fig-tree} (taken from \cite{Shmerkin11}) for an example when $d=2$ and $\ell=3$.
\item Cubes which do not contain $\e$-holes of (relative) level $\ell$ are split into the subcubes of first level, i.e. $\Q_1(Q)$.
\end{itemize}

 \begin{figure}
    \centering
    \includegraphics[width=0.5\textwidth]{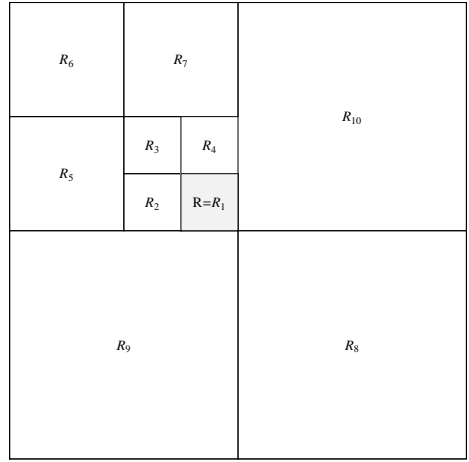}
    \caption{The correct way to subdivide a square with an $\e$-hole $R$.}\label{fig-tree}
  \end{figure}

This turns out to be slightly more efficient, because we are tailoring the decomposition in order to maximize the chance of witnessing holes. The corresponding local entropy averages result, in which cubes are subdivided into a family of cubes of possibly varying sizes (which may depend on the cube) needs to take into account not only the entropy of the measure, but also the sizes of the cubes. See \cite[Section 2 and Theorem 3.1]{Shmerkin11} for the setting and the precise statement. Using this result and the subdivision procedure just outlined, the proof of the bound \eqref{eq:bound-small-porosity} runs along similar lines to the proof of Proposition \ref{prop:weaker-bound}, with some minor new technical complications.

\section{Generalizations and further connections}
\label{sec:generalizations}

Many connections and generalizations to the porosity ideas presented so far have been proposed, sometimes motivated by specific applications, and sometimes by geometric reasons. We briefly discuss some of them.

\subsection{Conical densities}
\label{subsec:conical-densities}

A well-studied problem in geometric measure theory consists in understanding how a measure $\mu\in\mathcal{M}_d$ is distributed inside small cones (more precisely, the intersection of a cone based at a point $x$ with a small ball centered at $x$). In the classical case, the measure $\mu$ is the restriction of $s$-dimensional Hausdorff measure $\mathcal{H}^s$ to a set $E$ with $\mathcal{H}^s(E)<\infty$, but more recently packing measures as well as general measures were also considered.  It turns out that under suitable conditions, the answer is ``fairly well distributed'', meaning that all cones receive a positive proportion of the mass, which is uniform over all cones of a given amplitude. We illustrate this with a very general recent result of K\"{a}enm\"{a}ki et al \cite{KRS11}.

Recall that $G(d,k)$ denotes the Grassmanian of all $k$-dimensional subspaces of $\R^d$. Given a small $0<\alpha\le 1$ (the opening of the cone), $V\in G(d,d-k)$ (the subspace giving the ``direction'' of the cone) and $r>0$ (the radius of the reference ball), we define the cone
\[
X(x,r,V,\gamma) = \{y\in B(x,r): \dist(y-x,V) <\gamma|y-x| \}.
\]
Given $0<\gamma<1$, $\theta\in S^{d-1}$  we define the ``almost half-space'' (which is also an unbounded cone)
\[
H(x,\theta,\gamma)  = \{ y\in\R^d: (y-x)\cdot\theta >\gamma|y-x| \}.
\]

The theorem says that if $\ldim_P(\mu)>s$, then for any fixed $\gamma$, all ``non-symmetric cones'' $X(x,r,V,\gamma)\setminus H(x,\theta,\gamma)$ capture a uniformly positive proportion of the mass:

\begin{thm}{\cite[Theorem 5.1]{KRS11}} \label{thm:conical-density}
Let $\mu\in\mathcal{M}_d$ have lower packing dimension strictly larger than $k$ for some $k\in \{1,\ldots,d-1\}$. Then for any $0<\gamma\le 1$ there exists a constant $c$ (also depending on $d,k$ and $\ldim_P(\mu)$) such that
\[
\limsup_{r\to 0}\inf_{\theta\in S^{d-1}, V\in G(d,d-k)} \frac{\mu(X(x,r,V,\gamma)\setminus H(x,\theta,\gamma)}{\mu(B(x,r))} > c
\]
for $\mu$-a.e. $x$.
\end{thm}

For other recent progress on conical densities see \cite{KaenmakiSuomala08, CKRS10, Kaenmaki10, KaenmakiSuomala11}

In spirit, results such as this are not too far from dimension bounds for porous measures. Indeed, Theorem \ref{thm:conical-density} says that if $\ldim_P(\mu)>k$ then $\mu$ cannot have ``conical holes'' (or rather, ``non-symmetric $k$-dimensional conical holes''). It is in fact possible to use results on conical densities to prove upper bounds for the dimension of porous sets and measures in the large porosity situation. The idea is as follows. Suppose that $\mu\in\mathcal{P}_d$ satisfies
\[
B(z,\alpha r) \subset B(x,r) \quad\text{and}\quad \mu(B(z,r))<\e \mu(B(x,r))
\]
for some $x,z\in\R^d$, small $r>0, \e>0$ and $\alpha\in (0,1/2)$. Then, letting $V\subset\R^d$ the line with direction $z-x$, a geometric argument shows that
\[
X(x,r,V,\gamma) \subset B(x,2(1-2\alpha)r) \cup B(z,\alpha r),
\]
where the opening $\gamma$ depends on $\alpha$. If $\alpha$ is close to $1/2$, this says that the measure of the cone $X(x,r,V,\gamma)$ is just slightly larger than that of a ball $B(x,r')$ with $r'$ much smaller than $r$. On the other hand, the measure of the cone $X(x,r,V,\gamma)$ is at least a positive proportion of the measure of $B(x,r)$ under suitable assumptions (such as those of Theorem \ref{thm:conical-density}). Combining these facts carefully leads to a dimension bound for $\mu$.

This method was introduced by Mattila \cite{Mattila88}, who gave the first results on the dimension of porous sets. It was also used by K\"{a}enm\"{a}ki and Suomala in \cite{KaenmakiSuomala08} and \cite{KaenmakiSuomala11} to bound the dimension of $k$-porous sets and measures (to be defined below). However, in this way one gets poor quantitative results, and in order to obtain sharp bounds a direct approach to porosity is more effective.

\subsection{$k$-porosity}
\label{subsec:k-porosity}

We observed earlier that porosity does not allow to distinguish between points and sets as large as a hyperplane. Since one of the goals of the definition of porosity is to give a means to distinguish between the size sets of zero measure, it is desirable to have a concept of porosity type that is able to differentiate, at least, between subspaces of different dimensions. Such concept was introduced by K\"{a}enm\"{a}ki and Suomala in \cite{KaenmakiSuomala11}. Roughly speaking, a set is $k$-porous if inside a reference ball one can find holes in $k$ mutually orthogonal directions.

More precisely, given a set $E\subset\R^d$, $x\in\R^d$ and $r>0$, let
\begin{align*}
\por_k(E,x,r) = \sup\{\alpha: & \text{there are } z_1,\ldots, z_k \text{ such that } \\
& B(z_i,\alpha r) \subset B(x,r)\setminus E \text{ for all }i=1,\ldots, k \text{ and }\\
& (z_i-x)\cdot (z_j-x) = 0 \text{ for } i\neq j \}.
\end{align*}

As for usual porosity, we then define
\begin{align*}
\por_k(E,x) &= \liminf_{r\to 0} \por(E,x,r),\\
\por_k(E) &= \inf_{x\in E} \por(E,x).
\end{align*}
We can define $k$-porosity for measures following the usual procedure; the precise definition is left to the reader.

It is easy to see that if $V\subset\R^n$ is a subspace of dimension $\ell$, then $\por_k(V)=1/2$ for $k=1,\ldots,n-\ell$, and $\por_k(V)=0$ for $k=n-\ell+1,\ldots, n$. Hence $k$-porosity does allow us to detect the dimension of subspaces and, in general, provides a finer way to distinguish between fractal subsets of $\R^d$.

In light of Theorem \ref{thm:main}, it is natural to ask what are the corresponding bounds for $k$-porous measures (or sets). The idea of $k$ porosity is somewhat artificial in the small porosity context; in fact, one cannot give any better dimension bounds than \eqref{eq:bound-small-porosity}. In the large porosity case, K\"{a}enm\"{a}ki and Suomala in \cite{KaenmakiSuomala11} proved and applied conical density results to conclude that there is a function $g_{n,k}(\alpha)$ with $g_{n,k}(\alpha)\to 0$ as $\alpha\to 1/2$, such that
\[
\dim_H(E) \le n-k - g_{n,k}(\alpha) \quad\text{for all } E\in \R^d \text{ such that } \por_k(E)\ge \alpha.
\]
A $k$-plane shows that the limit $n-k$ as $\alpha\to 1/2$ is optimal. The sharp function $g_{n,k}$, up to multiplicative constants, was identified in \cite[Corollary 2.6]{JJKS05}: they show that if $\por_k(E)=s$, then
\[
\dim_P(E) \le n - k - \frac{C_n}{|\log(1-2\rho)|}.
\]
(Compare with \eqref{eq:bound-large-porosity}). Finally, in \cite{KRS11}, the same bound was obtained for $k$-porous measures.

To the best of our knowledge, the concept of mean $k$-porosity has not been investigated. It appears that a combination of the ideas of \cite{KRS11} with the method of local entropy averages may yield dimension bounds in that context.

\subsection{Directed porosity}
\label{subsec:directed-porosity}

V. Chousionis \cite{Chousionis09} introduced and studied the idea of {\em directed porosity}. This is similar to porosity, except that the holes are required to lie in a fixed direction from the centre of the reference ball. More precisely, let $A\subset \R^d$, and let $V\in G(d,k)$ for some $1\le k<d$. We define the $V$-directed porosity in the usual steps as follows:
\begin{align*}
\por^V(E,x,r) &= \sup\{\alpha: \exists z\in x+V, B(z,\alpha r)\subset B(x,r)\setminus E \},\\
\por^V(E,x) &= \liminf_{r\to 0} \por^V(E,x,r),\\
\por^V(E) &= \inf_{x\in A} \por^V(E,x).
\end{align*}

We note that directed porosity is a stronger notion than porosity (in other words, $\por^V(A)\le \por(A)$), and it becomes stronger as $k$ becomes smaller. In terms of dimension, directed porosity behaves the same as porosity; one cannot get any better bounds since the known examples illustrating the sharpness of the estimates can be taken to be $V$-directed porous for any $V\in G(d,1)$.

We recall the definition of the well-known {\em open set condition} for an IFS $\{f_1,\ldots, f_m\}$ on $\R^d$: there exists a nonempty open set $U\subset\R^d$ such that $f_i(U)\subset U$ for all $i$, and the sets $f_i(U)$ are mutually disjoint. The open set condition allows some overlaps between the pieces $f_i(E)$ (where $E$ is the attractor), but it guarantees that they are ``small'' so that they can be handled in many cases.

We need to recall one more definition: a set $E\subset\R^d$ is called {\em $k$-purely unrectifiable} if $\mathcal{H}^k(E\cap M)=0$ for any $C^1$ $k$-manifold $M\subset\R^d$.

Chousionis proved the following result on the directed porosity of self-similar sets:

\begin{thm}{\cite[Theorem 1.2 and Corollary 1.3]{Chousionis09}} \label{thm:chousionis}
Let $E\subset\R^d$ be a self-similar set with the open set condition. Then:
\begin{enumerate}
\item For any $k\in \{1,\ldots, d-1\}$ with $\dim_H(E)\le k$, $E$ is $V$-directed porous for all but at most one $V\in G(d,k)$.
\item If $E$ is in addition $k$-purely unrectifiable, then $E$ is $V$-directed porous for all $V\in G(d,k)$.
\end{enumerate}
\end{thm}
In fact, the theorem holds not only for self-similar sets, but for attractors of finite conformal iterated function systems under some natural conditions, see \cite[Section 2]{Chousionis09}.

The motivation for the study of directed porosity in \cite{Chousionis09} was the problem of convergence of truncated singular integrals with respect to general measures. In \cite[Theorem 1.4]{Chousionis09} it is proved that for a very wide class of antisymmetric kernels $K$ in $\R^d$ (including Riesz kernels), the truncated singular integrals
\[
T_\e(f)(x) = \int_{|x-y|>\e} K(x-y) f(y)d\mu(y)
\]
converge, as $\e\to 0$, for $f$ in the dense subspace of $L^2$ generated by the characteristic functions of balls, provided the support of $\mu$ is $V$-directed porous for all $V\in G(d,d-1)$, and $\mu$ satisfies the growth condition $\mu(B(x,r)) \le C r^{d-1}$ for all $x\in\supp\mu$, $r>0$. We note that one cannot have convergence in all of $L^2$ for general kernels (there are counterexamples for the $1$-dimensional Riesz kernel in $\R^2$).

\medskip

\textbf{Ackowledgments}. I am grateful to Antti K\"{a}enm\"{a}ki and Ville Suomala for many enlightening discussions about porosity, a careful reading of the manuscript, and pointing out numerous relevant references.


\end{document}